\documentclass[11pt]{amsart}
\usepackage{amssymb}
\usepackage{amsmath}
\usepackage[active]{srcltx}
\usepackage{t1enc}
\usepackage[latin2]{inputenc}
\usepackage{verbatim}
\usepackage{amsmath,amsfonts,amssymb,amsthm}
\usepackage[mathcal]{eucal}
\usepackage{enumerate}
\usepackage[centertags]{amsmath}
\usepackage{graphics}

\newtheorem{theorem}{Theorem}

\newtheorem{lemma}{Lemma}

\newtheorem{corollary}{Corollary}

\begin{document}
\author{I. Blahota and G. Tephnadze}
\title[Fej\'er means]{Strong convergence theorem for Vilenkin-Fej\'er means }
\address{I. Blahota, Institute of Mathematics and Computer Sciences, College
of Ny\'\i regyh\'aza, P.O. Box 166, Ny\'\i r\-egyh\'aza, H-4400, Hungary}
\email{blahota@nyf.hu}
\address{G. Tephnadze, Department of Mathematics, Faculty of Exact and
Natural Sciences, Tbilisi State University, Chavchavadze str. 1, Tbilisi
0128, Georgia and Department of Engineering Sciences and Mathematics, Lule%
\aa {} University of Technology, SE-971 87, Lule\aa {}, Sweden}
\email{giorgitephnadze@gmail.com}
\thanks{The research was supported by project TÁ%
MOP-4.2.2.A-11/1/KONV-2012-0051 and by Shota Rustaveli National Science
Foundation grant no.52/54 (Bounded operators on the martingale Hardy spaces).%
}
\date{}
\maketitle

\begin{abstract}
As main result we prove strong convergence theorems of Vilenkin-Fej\'er
means when $0<p\leq 1/2.$
\end{abstract}

\date{}



\section{Introduction}

It is well-known that Vilenkin system does not form basis in the space $%
L_{1}\left( G_{m}\right) .$ Moreover, there is a function in the Hardy space
$H_{1}\left( G_{m}\right) ,$ such that the partial sums of $f$ \ are not
bounded in $L_{1}$-norm. However, in Gát \cite{gat1} the following strong
convergence result was obtained for all $f\in H_{1}:$%
\begin{equation*}
\underset{n\rightarrow \infty }{\lim }\frac{1}{\log n}\overset{n}{\underset{%
k=1}{\sum }}\frac{\left\Vert S_{k}f-f\right\Vert _{1}}{k}=0,
\end{equation*}%
where $S_{k}f$ denotes the $k$-th partial sum of the Vilenkin-Fourier series
of $f.$ (For the trigonometric analogue see in Smith \cite{sm}, for the
Walsh-Paley system in Simon \cite{Si3}). Simon \cite{Si4} (see also \cite%
{tep6}) proved that there exists an absolute constant $c_{p},$ depending
only on $p,$ such that
\begin{equation}
\frac{1}{\log ^{\left[ p\right] }n}\overset{n}{\underset{k=1}{\sum }}\frac{%
\left\Vert S_{k}f\right\Vert _{p}^{p}}{k^{2-p}}\leq c_{p}\left\Vert
f\right\Vert _{H_{p}}^{p},\text{ \ \ \ }\left( 0<p\leq 1\right)  \label{1cc}
\end{equation}%
for all $f\in H_{p}$ and $n\in \mathbb{P}_{+},$ where $\left[ p\right] $
denotes integer part of $p.$ In \cite{tep4} it was proved that sequence $%
\left\{ 1/k^{2-p}\right\} _{k=1}^{\infty }$ $\left( 0<p<1\right) $ \ in (\ref%
{1cc}) are given exactly.

Weisz \cite{We2} considered the norm convergence of Fej\'er means of
Walsh-Fourier series and proved the following:

\textbf{Theorem W1\ (Weisz).} Let $p>1/2$ and $f\in H_{p}.$ Then there
exists an absolute constant $c_{p}$, depending only on $p$, such that for
all $f\in H_{p}$ and $k=1,2,\dots$
\begin{equation*}
\left\Vert \sigma _{k}f\right\Vert _{p}\leq c_{p}\left\Vert f\right\Vert
_{H_{p}}.
\end{equation*}%
Theorem W1 implies that%
\begin{equation*}
\frac{1}{n^{2p-1}}\overset{n}{\underset{k=1}{\sum }}\frac{\left\Vert \sigma
_{k}f\right\Vert _{p}^{p}}{k^{2-2p}}\leq c_{p}\left\Vert f\right\Vert
_{H_{p}}^{p},\text{ \ \ \ }\left( 1/2<p<\infty,\ n=1,2,\dots \right).
\end{equation*}

If Theorem W1 holds for $0<p\leq 1/2,$ then we would have%
\begin{equation}
\frac{1}{\log ^{\left[ 1/2+p\right] }n}\overset{n}{\underset{k=1}{\sum }}%
\frac{\left\Vert \sigma _{k}f\right\Vert _{p}^{p}}{k^{2-2p}}\leq
c_{p}\left\Vert f\right\Vert _{H_{p}}^{p},\text{ \ \ \ }\left( 0<p\leq 1/2,\
n=2,3,\dots \right).  \label{2cc}
\end{equation}

However, in \cite{tep1} it was proved that the assumption $p>1/2$ in Theorem
W1 is essential. In particular, the following is true:

\textbf{Theorem T1. }There exists a martingale $f\in H_{1/2}$ such that
\begin{equation*}
\sup_{n}\left\Vert \sigma _{n}f\right\Vert _{1/2}=+\infty.
\end{equation*}

For the Walsh system in \cite{tep5} it was proved that (\ref{2cc}) holds,
though Theorem T1 is not true for $0<p<1/2.$

As main result we generalize inequality (\ref{2cc}) for bounded Vilenkin
systems$.$

The results for summability of Fejér means of Walsh-Fourier series can be
found in \cite{BGG,BGG2,Fu}, \cite{GoAMH,GoPubl,GNCz,PS,Sc,SiW,Si2}.

\section{Definitions and Notations}

Let $\mathbb{P}_{+}$ denote the set of the positive integers, $\mathbb{P}:=%
\mathbb{P}_{+}\cup \{0\}.$

Let $m:=(m_{0,}m_{1},\dots)$ denote a sequence of the positive integers not
less than 2.

Denote by
\begin{equation*}
Z_{m_{k}}:=\{0,1,\dots,m_{k}-1\}
\end{equation*}
the additive group of integers modulo $m_{k}.$

Define the group $G_{m}$ as the complete direct product of the group $%
Z_{m_{j}}$ with the product of the discrete topologies of $Z_{m_{j}}$ $^{,}$%
s.

The direct product $\mu $ of the measures
\begin{equation*}
\mu _{k}\left( \{j\}\right):=1/m_{k}\text{ \qquad }(j\in Z_{m_{k}})
\end{equation*}
is the Haar measure on $G_{m_{\text{ }}}$with $\mu \left( G_{m}\right) =1.$

\textbf{In this paper we discuss bounded Vilenkin groups only, that is }%
\begin{equation*}
\sup_{n}m_{n}<\infty .
\end{equation*}

The elements of $G_{m}$ are represented by sequences
\begin{equation*}
x:=(x_{0},x_{1},\dots ,x_{k},\dots )\qquad \left( \text{ }x_{k}\in
Z_{m_{k}}\right) .
\end{equation*}

It is easy to give a base for the neighbourhood of $G_{m}$
\begin{equation*}
I_{0}\left( x\right):=G_{m},
\end{equation*}%
\begin{equation*}
I_{n}(x):=\{y\in G_{m}\mid y_{0}=x_{0},\dots,y_{n-1}=x_{n-1}\}\text{ }(x\in
G_{m},\text{ }n\in \mathbb{P})
\end{equation*}%
Denote $I_{n}:=I_{n}\left( 0\right) $ for $n\in \mathbb{P}$ and $\overline{%
I_{n}}:=G_{m}$ $\backslash $ $I_{n}$ $.$

Let

\begin{equation*}
e_{n}:=\left( 0,\dots,0,x_{n}=1,0,\dots\right) \in G_{m}\qquad \left( n\in
\mathbb{P}\right).
\end{equation*}

Denote
\begin{equation*}
I_{N}^{k,l}:=\left\{
\begin{array}{l}
I_{N}(0,\dots ,0,x_{k}\neq 0,0,\dots ,0,x_{l}\neq 0,x_{l+1},\dots
,x_{N-1},x_{N},x_{N+1},\dots ), \\
\ k<l<N, \\
I_{N}(0,\dots ,0,x_{k}\neq 0,0,\dots ,0,x_{N},x_{N+1},\dots ),\  \\
l=N.%
\end{array}%
\right.
\end{equation*}%
and
\begin{equation}
\overline{I_{N}}=\left( \overset{N-2}{\underset{k=0}{\bigcup }}\overset{N-1}{%
\underset{l=k+1}{\bigcup }}I_{N}^{k,l}\right) \bigcup \left( \underset{k=0}{%
\bigcup\limits^{N-1}}I_{N}^{k,N}\right) .  \label{2}
\end{equation}

If we define the so-called generalized number system based on $m$ in the
following way:
\begin{equation*}
M_{0}:=1,\text{ \qquad }M_{k+1}:=m_{k}M_{k\text{ }}\ \qquad (k\in \mathbb{P})
\end{equation*}%
then every $n\in \mathbb{P}$ can be uniquely expressed as
\begin{equation*}
n=\sum_{j=0}^{\infty }n_{j}M_{j},
\end{equation*}
where $n_{j}\in Z_{m_{j}}$ $~(j\in \mathbb{P})$ and only a finite number of $%
n_{j}`$s differ from zero. Let $\left\vert n\right\vert :=\max $ $\{j\in
\mathbb{P};$ $n_{j}\neq 0\}.$

For $n=\sum_{i=1}^{r}s_{i}M_{n_{i}}$, where $n_{1}>n_{2}>\dots >n_{r}\geq 0$
and $1\leq s_{i}<m_{n_{i}}$ for all $1\leq i\leq r$ we denote%
\begin{equation*}
\mathbb{A}_{0,2}=\left\{ n\in \mathbb{P}:\text{ }n=M_{0}+M_{2}+%
\sum_{i=1}^{r-2}s_{i}M_{n_{i}}\right\} .
\end{equation*}

The norm (or quasi-norm) of the space $L_{p}(G_{m})$ is defined by \qquad
\qquad \thinspace\
\begin{equation*}
\left\Vert f\right\Vert _{p}:=\left( \int_{G_{m}}\left\vert f(x)\right\vert
^{p}d\mu (x)\right) ^{1/p}\qquad \left( 0<p<\infty \right) .
\end{equation*}

The space $L_{p,\infty }\left( G_{m}\right) $ consists of all measurable
functions $f$ for which

\begin{equation*}
\left\Vert f\right\Vert _{L_{p},\infty }^{p}:=\underset{\lambda >0}{\sup }%
\lambda ^{p}\mu \left\{ f>\lambda \right\} <+\infty .
\end{equation*}

Next, we introduce on $G_{m}$ an orthonormal system which is called the
Vilenkin system.

At first define the complex valued function $r_{k}\left( x\right)
:G_{m}\rightarrow \mathbb{C},$ the generalized Rademacher functions as
\begin{equation*}
r_{k}\left( x\right) :=\exp \left( 2\pi \imath x_{k}/m_{k}\right) \text{
\qquad }\left( \imath ^{2}=-1,\text{ }x\in G_{m},\text{ }k\in \mathbb{P}%
\right) .
\end{equation*}

It is known that%
\begin{equation}
\sum_{k=0}^{m_{n}-1}r_{n}^{k}\left( x\right) =\left\{
\begin{array}{ll}
m_{n}, & \text{\thinspace \thinspace }x_{n}=0, \\
0, & \text{ \thinspace }x_{n}\neq 0,%
\end{array}%
\right.  \label{1a}
\end{equation}

Now define the Vilenkin system $\psi :=(\psi _{n}:n\in \mathbb{P})$ on $%
G_{m} $ as:
\begin{equation*}
\psi _{n}:=\prod_{k=0}^{\infty }r_{k}^{n_{k}}\left( x\right) \text{ \qquad }%
\left( n\in \mathbb{P}\right) .
\end{equation*}

Specially, we call this system the Walsh-Paley one if $m\equiv 2.$

The Vilenkin system is orthonormal and complete in $L_{2}\left( G_{m}\right)
\,$\cite{AVD,Vi}.

Now we introduce analogues of the usual definitions in Fourier-analysis. If $%
f\in L_{1}\left( G_{m}\right) $ we can establish the the Fourier
coefficients, the partial sums of the Fourier series, the Fejér means, the
Dirichlet and Fejér kernels with respect to the Vilenkin system in the usual
manner:
\begin{eqnarray*}
\widehat{f}\left( n\right) &:&=\int_{G_{m}}f\overline{\psi }_{n}d\mu ,\text{%
\ \ \ \ }(n\in \mathbb{P}_{+}) \\
S_{n}f &:&=\sum_{k=0}^{n-1}\widehat{f}\left( k\right) \psi _{k},\text{ \ \ }%
(n\in \mathbb{P}_{+}), \\
\sigma _{n}f &:&=\frac{1}{n}\sum_{k=1}^{n}S_{k}f,\text{ \ \ \ \ }\left( n\in
\mathbb{P}_{+}\right) , \\
D_{n} &:&=\sum_{k=0}^{n-1}\psi _{k},\text{\ \ \ \ \ \ \ \ \ \ }\left( n\in
\mathbb{P}_{+}\right) , \\
K_{n} &:&=\frac{1}{n}\overset{n}{\underset{k=1}{\sum }}D_{k},\text{ \
\thinspace\ \ }\left( n\in \mathbb{P}_{+}\right) .
\end{eqnarray*}

Recall that
\begin{equation}
\quad \hspace*{0in}D_{M_{n}}\left( x\right) =\left\{
\begin{array}{ll}
M_{n}, & \text{if\thinspace \thinspace }x\in I_{n}, \\
0, & \text{\thinspace if \thinspace \thinspace }x\notin I_{n},%
\end{array}%
\right.  \label{3}
\end{equation}%
and%
\begin{equation}
\quad \hspace*{0in}D_{n}\left( x\right) =\psi _{n}\left( x\right)
\sum_{j=0}^{\infty }D_{M_{j}}\left( x\right)
\sum_{p=m_{j}-n_{j}}^{m_{j}-1}r_{j}^{p}.  \label{3a}
\end{equation}

It is well-known that
\begin{equation}
\sup_{n}\int_{G_{m}}\left\vert K_{n}\left( x\right) \right\vert d\mu \left(
x\right) \leq c<\infty .  \label{4}
\end{equation}%
\vspace{0pt}

The $\sigma$-algebra generated by the intervals $\left\{ I_{n}\left(
x\right):x\in G_{m}\right\} $ will be denoted by $\digamma _{n}$ $\left(
n\in \mathbb{P}\right).$ Denote by $f=\left( f^{\left( n\right) },n\in
\mathbb{P}\right) $ a martingale with respect to $\digamma _{n}$ $\left(
n\in \mathbb{P}\right)$ (for details see e.g. \cite{We1}). The maximal
function of a martingale $f$ is defined by \qquad
\begin{equation*}
f^{*}=\sup_{n\in \mathbb{P}}\left| f^{\left( n\right) }\right|.
\end{equation*}

In case $f\in L_{1}(G_{m}),$ then it is easy to show that the sequence $%
\left( S_{M_{n}}\left( f\right) :n\in \mathbb{P}\right) $ is a martingale.
Moreover, the maximal functions are also be given by
\begin{equation*}
f^{\ast }\left( x\right) =\sup_{n\in \mathbb{P}}\frac{1}{\left\vert
I_{n}\left( x\right) \right\vert }\left\vert \int_{I_{n}\left( x\right)
}f\left( u\right) \mu \left( u\right) \right\vert
\end{equation*}

For $0<p<\infty $ the Hardy martingale spaces $H_{p}$ $\left( G_{m}\right) $
consist of all martingales for which
\begin{equation*}
\left\| f\right\| _{H_{p}}:=\left\| f^{*}\right\| _{p}<\infty.
\end{equation*}

If $f=\left( f^{\left( n\right) },n\in \mathbb{P}\right) $ is martingale
then the Vilenkin-Fourier coefficients must be defined in a slightly
different manner: $\qquad \qquad $
\begin{equation*}
\widehat{f}\left( i\right) :=\lim_{k\rightarrow \infty
}\int_{G_{m}}f^{\left( k\right) }\left( x\right) \overline{\psi }_{i}\left(
x\right) d\mu \left( x\right) .
\end{equation*}

The Vilenkin-Fourier coefficients of $f\in L_{1}\left( G_{m}\right) $ are
the same as those of the martingale $\left( S_{M_{n}}\left( f\right) :n\in
\mathbb{P}\right) $ obtained from $f$.

A bounded measurable function $a$ is p-atom, if there exist a dyadic
interval $I$, such that%
\begin{equation*}
\int_{I}ad\mu =0,\text{ \ \ }\left\Vert a\right\Vert _{\infty }\leq \mu
\left( I\right) ^{-1/p},\text{ \ \ supp}\left( a\right) \subset I.\qquad
\end{equation*}
\qquad

\section{Formulation of Main Result}

\begin{theorem}
Let $0<p\leq 1/2$. Then there exists an absolute constant $c_{p}>0$,
depending only on $p$, such that for all $f\in H_{p}$ and $n=2,3,\dots $
\begin{equation*}
\frac{1}{\log ^{\left[ 1/2+p\right] }n}\overset{n}{\underset{k=1}{\sum }}%
\frac{\left\Vert \sigma _{k}f\right\Vert _{p}^{p}}{k^{2-2p}}\leq
c_{p}\left\Vert f\right\Vert _{H_{p}}^{p},
\end{equation*}%
where $\left[ x\right] $ denotes integer part of $x.$
\end{theorem}

\begin{corollary}
Let $f\in H_{1/2}.$ Then
\begin{equation*}
\frac{1}{\log n}\overset{n}{\underset{k=1}{\sum }}\frac{\left\Vert \sigma
_{k}f-f\right\Vert _{1/2}^{1/2}}{k}\rightarrow 0,\text{ as }n\rightarrow
\infty.
\end{equation*}
\end{corollary}

\begin{theorem}
Let $0<p<1/2$ and $\Phi :\mathbb{P}_{+}\rightarrow \lbrack 1,$ $\infty )$ is
any nondecreasing function, satisfying the conditions $\Phi \left( n\right)
\uparrow \infty $ and
\begin{equation}
\overline{\underset{n\rightarrow \infty }{\lim }}\frac{n^{2-2p}}{\Phi \left(
n\right) }=\infty .  \label{cond}
\end{equation}%
Then there exists a martingale $f\in H_{p},$ such that
\begin{equation*}
\underset{k=1}{\overset{\infty }{\sum }}\frac{\left\Vert \sigma
_{k}f\right\Vert _{L_{p,\infty }}^{p}}{\Phi \left( k\right) }=\infty .
\end{equation*}
\end{theorem}

\section{Auxiliary Propositions}

\begin{lemma}
\cite{We3} (see also \cite{We3}) A martingale $f=\left( f^{\left( n\right)
},n\in \mathbb{P}\right) $ is in $H_{p}\left( 0<p\leq 1\right) $ if and only
if there exist a sequence $\left( a_{k},k\in \mathbb{P}\right) $ of p-atoms
and a sequence $\left( \mu _{k},k\in \mathbb{P}\right) $ of a real numbers
such that for every $n\in \mathbb{P}$
\end{lemma}

\begin{equation}
\qquad \sum_{k=0}^{\infty }\mu _{k}S_{M_{n}}a_{k}=f^{\left( n\right) },
\label{6}
\end{equation}

\begin{equation*}
\qquad \sum_{k=0}^{\infty }\left| \mu _{k}\right| ^{p}<\infty.
\end{equation*}
Moreover, $\left\| f\right\| _{H_{p}}\backsim \inf \left( \sum_{k=0}^{\infty
}\left| \mu _{k}\right| ^{p}\right) ^{1/p},$ where the infimum is taken over
all decomposition of $f$ of the form (\ref{6}).

\begin{lemma}
\cite{gat} Let $n>t,$ $t,n\in \mathbb{P},$ $x\in I_{t}\backslash $ $I_{t+1}$%
. Then
\end{lemma}

\begin{equation*}
K_{M_{n}}\left( x\right) =\left\{
\begin{array}{ll}
0, & \text{if }x-x_{t}e_{t}\notin I_{n}, \\
\frac{M_{t}}{1-r_{t}\left( x\right) }, & \text{if }x-x_{t}e_{t}\in I_{n}.%
\end{array}%
\right.
\end{equation*}

\begin{lemma}
\cite{tep2, tep3} \ Let $x\in I_{N}^{k,l},$ $k=0,\dots,N-2,$ $%
l=k+1,\dots,N-1.$ Then
\end{lemma}

\begin{equation*}
\int_{I_{N}}\left\vert K_{n}\left( x-t\right) \right\vert d\mu \left(
t\right) \leq \frac{cM_{l}M_{k}}{nM_{N}},\text{ \ \ when }n\geq M_{N}.
\end{equation*}

Let $x\in I_{N}^{k,N},$ $k=0,\dots,N-1.$ Then

\begin{equation*}
\int_{I_{N}}\left\vert K_{n}\left( x-t\right) \right\vert d\mu \left(
t\right) \leq \frac{cM_{k}}{M_{N}},\qquad \text{when }n\geq M_{N}.
\end{equation*}%
\qquad

\begin{lemma}
Let $n=\sum_{i=1}^{r}s_{i}M_{n_{i}}$, where $n_{1}>n_{2}>\dots >n_{r}\geq 0$
and $1\leq s_{i}<m_{n_{i}}$ for all $1\leq i\leq r$ as well as $%
n^{(k)}=n-\sum_{i=1}^{k}s_{i}M_{n_{i}}$, where $0<k\leq r$. Then
\begin{equation*}
nK_{n}=\sum_{k=1}^{r}\left( \prod_{j=1}^{k-1}r_{n_{j}}^{s_{j}}\right)
s_{k}M_{n_{k}}K_{s_{k}M_{n_{k}}}+\sum_{k=1}^{r-1}\left(
\prod_{j=1}^{k-1}r_{n_{j}}^{s_{j}}\right) n^{(k)}D_{s_{k}M_{n_{k}}}.
\end{equation*}
\end{lemma}

\begin{proof}
It is easy to see that if $k,s,n\in \mathbb{P},\ 0\leq k<M_{n}$, then
\begin{equation*}
D_{k+sM_{n}}=D_{sM_{n}}+\sum_{i=sM_{n}}^{sM_{n}+k-1}\psi
_{i}=D_{sM_{n}}+\sum_{i=0}^{k-1}\psi _{i+sM_{n}}=D_{sM_{n}}+r_{n}^{s}D_{k}.
\end{equation*}%
With help of this fact we get
\begin{eqnarray*}
nK_{n}
&=&\sum_{k=1}^{n}D_{k}=\sum_{k=1}^{s_{1}M_{n_{1}}}D_{k}+%
\sum_{k=s_{1}M_{n_{1}}+1}^{n}D_{k} \\
&=&s_{1}M_{n_{1}}K_{s_{1}M_{n_{1}}}+\sum_{k=1}^{n^{(1)}}D_{k+s_{1}M_{n_{1}}}
\\
&=&s_{1}M_{n_{1}}K_{s_{1}M_{n_{1}}}+\sum_{k=1}^{n^{(1)}}\left(
D_{s_{1}M_{n_{1}}}+r_{n_{1}}^{s_{1}}D_{k}\right) \\
&=&s_{1}M_{n_{1}}K_{s_{1}M_{n_{1}}}+n^{(1)}D_{s_{1}M_{n_{1}}}+r_{n_{1}}^{s_{1}}n^{(1)}K_{n^{(1)}}.
\end{eqnarray*}%
If we unfold $n^{(1)}K_{n^{(1)}}$ in similar way, we have
\begin{equation*}
n^{(1)}K_{n^{(1)}}=s_{2}M_{n_{2}}K_{s_{2}M_{n_{2}}}+n^{(2)}D_{s_{2}M_{n_{2}}}+r_{n_{2}}^{s_{2}}n^{(2)}K_{n^{(2)}},
\end{equation*}%
so
\begin{eqnarray*}
nK_{n}
&=&s_{1}M_{n_{1}}K_{s_{1}M_{n_{1}}}+r_{n_{1}}^{s_{1}}s_{2}M_{n_{2}}K_{s_{2}M_{n_{2}}}+r_{n_{1}}^{s_{1}}r_{n_{2}}^{s_{2}}n^{(2)}K_{n^{(2)}}
\\
&&+n^{(1)}D_{s_{1}M_{n_{1}}}+r_{n_{1}}^{s_{1}}n^{(2)}D_{s_{2}M_{n_{2}}}.
\end{eqnarray*}%
Using this method with $n^{(2)}K_{n^{(2)}},\dots ,n^{(r-1)}K_{n^{(r-1)}}$,
we obtain
\begin{eqnarray*}
nK_{n} &=&\sum_{k=1}^{r}\left( \prod_{j=1}^{k-1}r_{n_{j}}^{s_{j}}\right)
s_{k}M_{n_{k}}K_{s_{k}M_{n_{k}}}+\left(
\prod_{j=1}^{r}r_{n_{j}}^{s_{j}}\right) n^{(r)}K_{n^{(r)}} \\
&&+\sum_{k=1}^{r-1}\left( \prod_{j=1}^{k-1}r_{n_{j}}^{s_{j}}\right)
n^{(k)}D_{s_{k}M_{n_{k}}}.
\end{eqnarray*}%
According to $n^{(r)}=0$ it yields the statement of the Lemma 4.
\end{proof}

\begin{lemma}
\cite{blahota} Let $s,n\in \mathbb{P}$. Then
\begin{equation*}
D_{sM_{n}}=D_{M_{n}}\sum_{k=0}^{s-1}\psi
_{kM_{n}}=D_{M_{n}}\sum_{k=0}^{s-1}r_{n}^{k}.
\end{equation*}
\end{lemma}

\begin{lemma}
Let $s,t,n\in \mathbb{N},\ n>t,\ s<m_{n},\ x\in I_{t}\backslash I_{t+1} $.
If $x-x_{t}e_{t}\notin I_{n}$, then
\begin{equation*}
K_{sM_{n}}(x)=0.
\end{equation*}
\end{lemma}

\begin{proof}
In \cite{gat} G. G\'at proved similar statement to $K_{M_{n}}(x)=0$. We will
use his method. Let $x\in I_{t}\backslash I_{t+1}$. Using (\ref{3}) and (\ref%
{3a}) we have
\begin{eqnarray*}
sM_{n}K_{sM_{n}}\left( x\right)
&=&\sum_{k=1}^{sM_{n}}D_{k}(x)=\sum_{k=1}^{sM_{n}}\psi _{k}(x)\left(
\sum_{j=0}^{t-1}k_{j}M_{j}+M_{t}\sum_{i=m_{t}-k_{t}}^{m_{t}-1}r_{t}^{i}(x)%
\right) \\
&=&\sum_{k=1}^{sM_{n}}\psi
_{k}(x)\sum_{j=0}^{t-1}k_{j}M_{j}+\sum_{k=1}^{sM_{n}}\psi
_{k}(x)M_{t}\sum_{i=m_{t}-k_{t}}^{m_{t}-1}r_{t}^{i}(x) \\
&=&J_{1}+J_{2}.
\end{eqnarray*}%
Let $k:=\sum_{j=0}^{n}k_{j}M_{j}$. Applying (\ref{1a}) we get $%
\sum_{k_{t}=0}^{m_{t}-1}r_{t}^{k_{t}}(x)=0$, for $x\in I_{t}\backslash
I_{t+1}$. It follows that%
\begin{equation*}
J_{1}=\sum_{k_{0}=0}^{m_{0}-1}\cdots
\sum_{k_{t-1}=0}^{m_{t-1}-1}\sum_{k_{t+1}=0}^{m_{t+1}-1}\cdots
\sum_{k_{n-1}=0}^{m_{n-1}-1}\sum_{k_{n}=0}^{s-1}\left( \prod_{\substack{ l=0
\\ l\neq t}}^{n}r_{l}^{k_{l}}(x)\right)
\sum_{j=0}^{t-1}k_{j}M_{j}\sum_{k_{t}=0}^{m_{t}-1}r_{t}^{k_{t}}(x)=0.
\end{equation*}%
On the other hand
\begin{eqnarray*}
J_{2}&=&\sum_{k_{0}=0}^{m_{0}-1}\cdots
\sum_{k_{t-1}=0}^{m_{t-1}-1}\sum_{k_{t+1}=0}^{m_{t+1}-1}\cdots
\sum_{k_{n-1}=0}^{m_{n-1}-1}\sum_{k_{n}=0}^{s-1}\left( \prod_{\substack{ l=0
\\ l\neq t}}^{n}r_{l}^{k_{l}}(x)\right) M_{t}\sum_{i=0}^{k_{t}-1}r_{t}^{i}(x)
\\
&=&\prod_{\substack{ l=0  \\ l\neq t}}^{n-1}\left(
\sum_{k_{l}=0}^{m_{l}-1}r_{l}^{k_{l}}(x)\right) \left(
\sum_{k_{p}=0}^{s}r_{p}^{k_{p}}(x)\right)
M_{t}\sum_{i=0}^{k_{t}-1}r_{t}^{i}(x).
\end{eqnarray*}%
Since $x-x_{t}e_{t}\notin I_{n}$, at least one of $%
\sum_{k_{l}=0}^{m_{l}-1}r_{l}^{k_{l}}(x)$ will be zero, if $l=p\neq t$ and $%
0\leq p\leq n-1$, that is $J_{2}=0$.
\end{proof}

\section{Proof of the Theorems}

\textbf{Proof of Theorem 1. }By Lemma 1, the proof of Theorem 1 will be
complete, if we show that with a constant $c_{p}$

\begin{equation*}
\frac{1}{\log ^{\left[ 1/2+p\right] }n}\overset{n}{\underset{k=1}{\sum }}%
\frac{\left\Vert \sigma _{k}a\right\Vert _{p}^{p}}{k^{2-2p}}\leq
c_{p}<\infty \ (n=2,3,\dots).
\end{equation*}%
for every p-atom $a,$ where $\left[ 1/2+p\right] $ denotes the integers part
of $1/2+p.$ We may assume that $a$ be an arbitrary p-atom with support$\ I$,
$\mu \left( I\right) =M_{N}^{-1}$ and $I=I_{N}.$ It is easy to see that $%
\sigma _{n}\left( a\right) =0,$ when $n\leq M_{N}$. Therefore we can suppose
that $n>M_{N}$.

Let $x\in I_{N}.$ Since $\sigma _{n}$ is bounded from $L_{\infty }$ to $%
L_{\infty }$ (the boundedness follows from (\ref{4})) and $\left\Vert
a\right\Vert _{\infty }\leq cM_{N}^{1/p}$ we obtain
\begin{equation*}
\int_{I_{N}}\left\vert \sigma _{m}a\left( x\right) \right\vert ^{p}d\mu
\left( x\right) \leq c\left\Vert a\right\Vert _{\infty }^{p}/M_{N}\leq
c_{p}<\infty ,\text{ \ }0<p\leq 1/2.
\end{equation*}%
Hence
\begin{eqnarray}
&&\frac{1}{\log ^{\left[ 1/2+p\right] }n}\overset{n}{\underset{m=1}{\sum }}%
\frac{\int_{I_{N}}\left\vert \sigma _{m}a\left( x\right) \right\vert
^{p}d\mu \left( x\right) }{m^{2-2p}}  \label{6aaaa} \\
&\leq &\frac{c}{\log ^{\left[ 1/2+p\right] }n}\overset{n}{\underset{m=1}{%
\sum }}\frac{1}{m^{2-2p}}\leq c_{p}<\infty .  \notag
\end{eqnarray}

It is easy to show that
\begin{eqnarray*}
&&\left\vert \sigma _{m}a\left( x\right) \right\vert \leq
\int_{I_{N}}\left\vert a\left( t\right) \right\vert \left\vert K_{m}\left(
x-t\right) \right\vert d\mu \left( t\right) \\
&\leq &\left\Vert a \right\Vert _{\infty }\int_{I_{N}}\left\vert K_{m}\left(
x-t\right) \right\vert d\mu \left( t\right) \leq
cM_{N}^{1/p}\int_{I_{N}}\left\vert K_{m}\left( x-t\right) \right\vert d\mu
\left( t\right).
\end{eqnarray*}

Let $x\in I_{N}^{k,l},\,0\leq k<l<N.$ Then from Lemma 3 we get
\begin{equation}
\left\vert \sigma _{m}a\left( x\right) \right\vert \leq \frac{%
cM_{l}M_{k}M_{N}^{1/p-1}}{m}.  \label{12}
\end{equation}

Let $x\in I_{N}^{k,N},\,0\leq k<N.$ Then from Lemma 3 we have

\begin{equation}
\left\vert \sigma _{m}a\left( x\right) \right\vert \leq cM_{k}M_{N}^{1/p-1}.
\label{12a}
\end{equation}

Since
\begin{equation*}
\overset{N-2}{\underset{k=0}{\sum }}1/M_{k}^{1-2p}\leq N^{\left[ 1/2+p\right]
},\text{ for }0<p\leq 1/2
\end{equation*}%
by combining (\ref{2}) and (\ref{12}-\ref{12a}) we obtain%
\begin{eqnarray}
&&\int_{\overline{I_{N}}}\left\vert \sigma _{m}a\left( x\right) \right\vert
^{p}d\mu \left( x\right)  \label{7aaa} \\
&=&\overset{N-2}{\underset{k=0}{\sum }}\overset{N-1}{\underset{l=k+1}{\sum }}%
\sum\limits_{x_{j}=0,j\in
\{l+1,...,N-1\}}^{m_{j-1}}\int_{I_{N}^{k,l}}\left\vert \sigma _{m}a\left(
x\right) \right\vert ^{p}d\mu \left( x\right)  \notag \\
&&+\overset{N-1}{\underset{k=0}{\sum }}\int_{I_{N}^{k,N}}\left\vert \sigma
_{m}a\left( x\right) \right\vert ^{p}d\mu \left( x\right)  \notag \\
&\leq &c\overset{N-2}{\underset{k=0}{\sum }}\overset{N-1}{\underset{l=k+1}{%
\sum }}\frac{m_{l+1}\dots m_{N-1}}{M_{N}}\frac{\left( M_{l}M_{k}\right)
^{p}M_{N}^{1-p}}{m^{p}}+\overset{N-1}{\underset{k=0}{\sum }}\frac{1}{M_{N}}%
M_{k}^{p}M_{N}^{1-p}  \notag \\
&\leq &\frac{cM_{N}^{1-p}}{m^{p}}\overset{N-2}{\underset{k=0}{\sum }}\overset%
{N-1}{\underset{l=k+1}{\sum }}\frac{\left( M_{l}M_{k}\right) ^{p}}{M_{l}}+%
\overset{N-1}{\underset{k=0}{\sum }}\frac{M_{k}^{p}}{M_{N}^{p}}  \notag \\
&=&\frac{cM_{N}^{1-p}}{m^{p}}\overset{N-2}{\underset{k=0}{\sum }}\frac{1}{%
M_{k}^{1-2p}}\overset{N-1}{\underset{l=k+1}{\sum }}\frac{M_{k}^{1-p}}{%
M_{l}^{1-p}}+\overset{N-1}{\underset{k=0}{\sum }}\frac{M_{k}^{p}}{M_{N}^{p}}
\notag \\
&\leq &\frac{cM_{N}^{1-p}N^{\left[ 1/2+p\right] }}{m^{p}}+c_{p}.  \notag
\end{eqnarray}

It is easy to show that

\begin{equation*}
\overset{n}{\underset{m=M_{N}+1}{\sum }}\frac{1}{m^{2-p}}\leq \frac{c}{%
M_{N}^{1-p}},\text{ for }0<p\leq 1/2.
\end{equation*}

By applying (\ref{6aaaa}) and (\ref{7aaa}) we get
\begin{eqnarray*}
&&\frac{1}{\log ^{\left[ 1/2+p\right] }n}\overset{n}{\underset{m=1}{\sum }}%
\frac{\left\Vert \sigma _{m}a\right\Vert _{p}^{p}}{m^{2-2p}} \\
&\leq &\frac{1}{\log ^{\left[ 1/2+p\right] }n}\overset{n}{\underset{m=M_{N}+1%
}{\sum }}\frac{\int_{\overline{I_{N}}}\left\vert \sigma _{m}a\left( x\right)
\right\vert ^{p}d\mu \left( x\right) }{m^{2-2p}} \\
&&+\frac{1}{\log ^{\left[ 1/2+p\right] }n}\overset{n}{\underset{m=M_{N}+1}{%
\sum }}\frac{\int_{I_{N}}\left\vert \sigma _{m}a\left( x\right) \right\vert
^{p}d\mu \left( x\right) }{m^{2-2p}} \\
&\leq &\frac{1}{\log ^{\left[ 1/2+p\right] }n}\overset{n}{\underset{m=M_{N}+1%
}{\sum }}\left( \frac{c_{p}M_{N}^{1-p}N^{\left[ 1/2+p\right] }}{m^{2-p}}+%
\frac{c_{p}}{m^{2-p}}\right) +c_{p} \\
&\leq &\frac{c_{p}M_{N}^{1-p}N^{\left[ 1/2+p\right] }}{\log ^{\left[ 1/2+p%
\right] }n}\overset{n}{\underset{m=M_{N}+1}{\sum }}\frac{1}{m^{2-p}}+\frac{1%
}{\log ^{\left[ 1/2+p\right] }n}\overset{n}{\underset{m=M_{N}+1}{\sum }}%
\frac{1}{m^{2-p}}+c_{p} \\
&\leq &c_{p}<\infty .
\end{eqnarray*}

which completes the proof of Theorem 1.

\textbf{Proof of Theorem 2. }Under condition (\ref{cond}) there exists a
sequence of increasing numbers $\left\{ n_{k}:\text{ }k\geq 0\right\} $,
such that
\begin{equation*}
\underset{k\rightarrow \infty }{\lim }\frac{cn_{k}^{2-2p}}{\Phi \left(
n_{k}\right) }=\infty .
\end{equation*}

\bigskip It is evident that for every $n_{k}$ there exists a positive
integer $\lambda _{k}$ such that%
\begin{equation*}
M_{\left\vert \lambda _{k}\right\vert +1}\leq n_{k}<M_{\left\vert \lambda
_{k}\right\vert +2}\leq \lambda M_{\left\vert n_{k}\right\vert +1},
\end{equation*}%
where $\lambda =\sup_{n}m_{n}$. Since $\Phi \left( n\right) $ is a
nondecreasing function we have
\begin{equation}
\overline{\underset{k\rightarrow \infty }{\lim }}\frac{M_{\left\vert
\lambda_{k}\right\vert +1}^{2-2p}}{\Phi \left( M_{\left\vert
\lambda_{k}\right\vert +1}\right) }\geq \underset{k\rightarrow \infty }{\lim
}\frac{cn_{k}^{2-2p}}{\Phi \left( n_{k}\right) }=\infty .  \label{12j}
\end{equation}

Applying (\ref{12j}) there exists a sequence $\left\{ \alpha _{k}:\text{ }%
k\geq 0\right\} \subset \left\{ \lambda _{k}:\text{ }k\geq 0\right\} $ such
that%
\begin{equation}
\left\vert \alpha _{k}\right\vert \geq 2,\text{ for \ }k\in \mathbb{P},
\label{22}
\end{equation}%
\begin{equation}
\underset{k\rightarrow \infty }{\lim }\frac{M_{\left\vert \alpha
_{k}\right\vert }^{1-p}}{\Phi ^{1/2}\left( M_{\left\vert \alpha
_{k}\right\vert +1}\right) }=\infty  \label{22w}
\end{equation}%
and%
\begin{equation}
\sum_{\eta =0}^{\infty }\frac{\Phi ^{1/2}\left( M_{\left\vert \alpha _{\eta
}\right\vert +1}\right) }{M_{\left\vert \alpha _{\eta }\right\vert }^{1-p}}%
=m_{\left\vert \alpha _{\eta }\right\vert }^{1-p}\sum_{\eta =0}^{\infty }%
\frac{\Phi ^{1/2}\left( M_{\left\vert \alpha _{\eta }\right\vert +1}\right)
}{M_{\left\vert \alpha _{\eta }\right\vert +1}^{1-p}}<c<\infty .  \label{121}
\end{equation}

Let \qquad
\begin{equation*}
f_{A}=\sum_{\left\{ k:\text{ }\left\vert \alpha _{k}\right\vert <A\right\}
}\lambda _{k}a_{k},
\end{equation*}%
where
\begin{equation*}
\lambda _{k}=\lambda \cdot \frac{\Phi ^{1/2p}\left( M_{\left\vert \alpha
_{k}\right\vert +1}\right) }{M_{\left\vert \alpha _{k}\right\vert }^{1/p-1}}
\end{equation*}%
and

\begin{equation*}
a_{k}=\frac{M_{\left\vert \alpha _{k}\right\vert }^{1/p-1}}{\lambda }\left(
D_{M_{\left\vert \alpha _{k}\right\vert }+1}-D_{M_{\left\vert \alpha
_{k}\right\vert }}\right) ,
\end{equation*}%
where $\lambda:= \sup_{n\in \mathbb{P}}m_{n}.$ Since

\begin{equation*}
S_{M_{n}}a_{k}=\left\{
\begin{array}{ll}
a_{k}, & \left\vert \alpha _{k}\right\vert <n, \\
0, & \left\vert \alpha _{k}\right\vert \geq n,%
\end{array}%
\right.
\end{equation*}%
and

\begin{equation*}
\text{supp}(a_{k})=I_{\left\vert \alpha _{k}\right\vert },\text{ \ \ \ }%
\int_{I_{\left\vert \alpha _{k}\right\vert }}a_{k}d\mu =0,\text{ \ \ }%
\left\Vert a_{k}\right\Vert _{\infty }\leq M_{\left\vert \alpha
_{k}\right\vert }^{1/p}=\left( \text{supp }a_{k}\right) ^{-1/p}
\end{equation*}%
if we apply Lemma 1 and (\ref{121}) we conclude that $f\in H_{p}.$

It is easy to show that

\begin{equation}
\widehat{f}(j)  \label{6aa}
\end{equation}

\begin{equation*}
=\left\{
\begin{array}{ll}
\Phi ^{1/2p}\left( M_{\left\vert \alpha _{k}\right\vert +1}\right) , & \text{%
if }j\in \left\{ M_{\left\vert \alpha _{k}\right\vert },\dots ,\text{ ~}%
M_{\left\vert \alpha _{k}\right\vert +1}-1\right\} ,\text{ }k=0,1,2,\dots ,
\\
0, & \text{if }j\notin \bigcup\limits_{k=0}^{\infty }\left\{ M_{\left\vert
\alpha _{k}\right\vert },\dots ,\text{ ~}M_{\left\vert \alpha
_{k}\right\vert +1}-1\right\} .\text{ }%
\end{array}%
\right.
\end{equation*}%
By using (\ref{6aa}) we can write that

\begin{equation}
\sigma _{_{\alpha _{k}}}f=\frac{1}{\alpha _{k}}\sum_{j=1}^{M_{\left\vert
\alpha _{k}\right\vert }}S_{j}f+\frac{1}{\alpha _{k}}\sum_{j=M_{\left\vert
\alpha _{k}\right\vert }+1}^{\alpha _{k}}S_{j}f=I+II.  \label{7aa}
\end{equation}

It is simple to show that

\begin{equation*}
S_{j}f=\left\{
\begin{array}{ll}
\Phi ^{1/2p}\left( M_{\left\vert \alpha _{0}\right\vert +1}\right) , & \text{%
if }M_{\left\vert \alpha _{0}\right\vert }<j\leq M_{\left\vert \alpha
_{0}\right\vert +1} \\
0, & \text{if }0\leq j\leq M_{\left\vert \alpha _{0}\right\vert }.\text{ }%
\end{array}%
\right.
\end{equation*}

Suppose that $M_{\left\vert \alpha _{s}\right\vert }<j\leq M_{\left\vert
\alpha _{s}\right\vert +1},$ for some $s=1,2,...,k.$ Then by applying (\ref%
{6aa}) we have that
\begin{equation}
S_{j}f=\sum_{v=0}^{M_{\left\vert \alpha _{s-1}\right\vert }}\widehat{f}%
(v)w_{v}+\sum_{v=M_{\left\vert \alpha _{s}\right\vert }+1}^{j-1}\widehat{f}%
(v)w_{v}  \label{27}
\end{equation}%
\begin{equation*}
=\sum_{\eta =0}^{s-1}\sum_{v=M_{\left\vert \alpha _{\eta }\right\vert
}}^{M_{\left\vert \alpha _{\eta }\right\vert +1}-1}\widehat{f}%
(v)w_{v}+\sum_{v=M_{\left\vert \alpha _{s}\right\vert }+1}^{j-1}\widehat{f}%
(v)w_{v}
\end{equation*}%
\begin{equation*}
=\sum_{\eta =0}^{s-1}\sum_{v=M_{\left\vert \alpha _{\eta }\right\vert
}}^{M_{\left\vert \alpha _{\eta }\right\vert +1}-1}\Phi ^{1/2p}\left(
M_{\left\vert \alpha _{\eta }\right\vert +1}\right) w_{v}+\Phi ^{1/2p}\left(
M_{\left\vert \alpha _{s}\right\vert +1}\right) \sum_{v=M_{\left\vert \alpha
_{s}\right\vert }+1}^{j-1}w_{v}
\end{equation*}%
\begin{equation*}
=\sum_{\eta =0}^{s-1}\Phi ^{1/2p}\left( M_{\left\vert \alpha _{\eta
}\right\vert +1}\right) \left( D_{M_{\left\vert \alpha _{\eta }\right\vert
+1}}-D_{M_{\left\vert \alpha _{\eta }\right\vert }}\right) +\Phi
^{1/2p}\left( M_{\left\vert \alpha _{s}\right\vert +1}\right) \left(
D_{_{j}}-D_{M_{\left\vert \alpha _{s}\right\vert }}\right) .
\end{equation*}

Let $M_{\left\vert \alpha _{s}\right\vert +1}<j\leq M_{\left\vert \alpha
_{s+1}\right\vert },$ for some $s=1,2,...,k.$ Analogously to (\ref{27}) we
get that
\begin{equation}
S_{j}f=\sum_{v=0}^{M_{\left\vert \alpha _{s}\right\vert +1}}\widehat{f}%
(v)w_{v}=\sum_{\eta =0}^{s}\Phi ^{1/2p}\left( M_{\left\vert \alpha _{\eta
}\right\vert +1}\right) \left( D_{M_{\left\vert \alpha _{\eta }\right\vert
+1}}-D_{M_{\left\vert \alpha _{\eta }\right\vert }}\right) .  \label{28}
\end{equation}

Let $x\in I_{2}^{0,1}=\left( x_{0}=1,\text{ }x_{1}=1,\text{ }x_{2},\dots
\right) .$ Since (see (\ref{3}) and Lemma 2)%
\begin{equation}
K_{M_{n}}\left( x\right) =D_{M_{n}}\left( x\right) =0,\text{ \ \ \ for \ \ }%
n\geq 2  \label{29}
\end{equation}%
from (\ref{22}) and (\ref{27})-(\ref{28}) we obtain that%
\begin{equation}
I=\frac{1}{n}\sum_{\eta =0}^{k-1}\Phi ^{1/2p}\left( M_{\left\vert \alpha
_{\eta }\right\vert +1}\right) \sum_{v=M_{\left\vert \alpha _{\eta
}\right\vert }+1}^{M_{\left\vert \alpha _{\eta }\right\vert +1}}D_{v}
\label{30a}
\end{equation}%
\begin{equation*}
=\frac{1}{n}\sum_{\eta =0}^{k-1}\Phi ^{1/2p}\left( M_{\left\vert \alpha
_{\eta }\right\vert +1}\right) \left( M_{\left\vert \alpha _{\eta
}\right\vert +1}K_{M_{\left\vert \alpha _{\eta }\right\vert +1}}\left(
x\right) -M_{\left\vert \alpha _{\eta }\right\vert }K_{M_{\left\vert \alpha
_{\eta }\right\vert }}\left( x\right) \right) =0.
\end{equation*}

By applying (\ref{27}), when $s=k$ \ in $II$ \ we get that

\begin{equation}
II=\frac{\alpha _{k}-M_{\left\vert n_{k}\right\vert }}{\alpha _{k}}%
\sum_{\eta =0}^{k-1}\Phi ^{1/2p}\left( M_{\left\vert \alpha _{\eta
}\right\vert +1}\right) \left( D_{M_{\left\vert \alpha _{\eta }\right\vert
+1}}-D_{M_{\left\vert \alpha _{\eta }\right\vert }}\right)  \label{9aa}
\end{equation}%
\begin{equation*}
+\frac{\Phi ^{1/2p}\left( M_{\left\vert n_{k}\right\vert +1}\right) }{\alpha
_{k}}\sum_{j=M_{\left\vert n_{k}\right\vert }+1}^{\alpha _{k}}\left(
D_{_{j}}-D_{M_{\left\vert n_{k}\right\vert }}\right) =II_{1}+II_{2}.
\end{equation*}

By using (\ref{29}) we have that

\begin{equation}
II_{1}=0,\text{ \ for }x\in I_{2}^{0,1}.  \label{10aaa}
\end{equation}

Let $\alpha _{k}\in \mathbb{A}_{0,2}$ and $x\in I_{2}^{0,1}$. Since $\alpha
_{k}-M_{\left\vert \alpha _{k}\right\vert }\in \mathbb{A}_{0,2}$ and%
\begin{equation*}
D_{j+M_{\left\vert \alpha _{k}\right\vert }}=D_{M_{\left\vert \alpha
_{k}\right\vert }}+w_{_{M_{\left\vert \alpha _{k}\right\vert }}}D_{j},\text{
when \thinspace \thinspace }j<M_{\left\vert \alpha _{k}\right\vert }
\end{equation*}%
By combining (\ref{3}) Lemmas 4 and 6 we obtain that%
\begin{eqnarray}
&&\left\vert II_{2}\right\vert =\frac{\Phi ^{1/2p}\left( M_{\left\vert
\alpha _{k}\right\vert +1}\right) }{\alpha _{k}}\left\vert
\sum_{j=1}^{\alpha _{k}-M_{\left\vert \alpha _{k}\right\vert }}\left(
D_{j+M_{\left\vert \alpha _{k}\right\vert }}\left( x\right)
-D_{M_{\left\vert \alpha _{k}\right\vert }}\left( x\right) \right)
\right\vert  \label{11aaa} \\
&=&\frac{\Phi ^{1/2p}\left( M_{\left\vert \alpha _{k}\right\vert +1}\right)
}{\alpha _{k}}\left\vert \sum_{j=1}^{\alpha _{k}-M_{\left\vert \alpha
_{k}\right\vert }}D_{_{j}}\left( x\right) \right\vert  \notag \\
&=&\frac{\Phi ^{1/2p}\left( M_{\left\vert \alpha _{k}\right\vert +1}\right)
}{\alpha _{k}}\left\vert \left( \alpha _{k}-M_{\left\vert \alpha
_{k}\right\vert }\right) K_{\alpha _{k}-M_{\left\vert \alpha _{k}\right\vert
}}\left( x\right) \right\vert  \notag \\
&=&\frac{\Phi ^{1/2p}\left( M_{\left\vert \alpha _{k}\right\vert +1}\right)
}{\alpha _{k}}\left\vert M_{0}K_{M_{0}}\right\vert \geq \frac{\Phi
^{1/2p}\left( M_{\left\vert \alpha _{k}\right\vert +1}\right) }{\alpha _{k}}.
\notag
\end{eqnarray}

Let $0<p<1/2,$ $n\in \mathbb{A}_{0,2}$ and $M_{\left\vert \alpha
_{k}\right\vert }<n<M_{\left\vert \alpha _{k}\right\vert +1}.$ By combining (%
\ref{7aa}-\ref{11aaa}) we have that

\begin{eqnarray*}
&&\left\Vert \sigma _{n}f\right\Vert _{L_{p,\infty }}^{p}\geq \frac{c\Phi
^{1/2}\left( M_{\left\vert \alpha _{k}\right\vert +1}\right) }{\alpha
_{k}^{p}}\mu \left\{ x\in I_{2}^{0,1}:\text{ }\left\vert II_{2}\right\vert
\geq \frac{c\Phi ^{1/2p}\left( M_{\left\vert \alpha _{k}\right\vert
+1}\right) }{\alpha _{k}}\right\} \\
&\geq &\frac{c\Phi ^{1/2}\left( M_{\left\vert \alpha _{k}\right\vert
+1}\right) }{\alpha _{k}^{p}}\mu \left\{ I_{2}^{0,1}\right\} \geq \frac{%
c\Phi ^{1/2}\left( M_{\left\vert \alpha _{k}\right\vert +1}\right) }{%
M_{\left\vert \alpha _{k}\right\vert +1}^{p}}.
\end{eqnarray*}

By using (\ref{22w}) we get that
\begin{eqnarray*}
&&\underset{n=1}{\overset{\infty }{\sum }}\frac{\left\Vert \sigma
_{n}f\right\Vert _{L_{p,\infty }}^{p}}{\Phi \left( n\right) }\geq \underset{%
\left\{ n\in \mathbb{A}_{0,2}:\text{ }M_{\left\vert \alpha _{k}\right\vert
}<n<M_{\left\vert \alpha _{k}\right\vert +1}\right\} }{\sum }\frac{%
\left\Vert \sigma _{n}f\right\Vert _{L_{p,\infty }}^{p}}{\Phi \left(
n\right) } \\
&\geq &\frac{1}{\Phi ^{1/2}\left( M_{\left\vert \alpha _{k}\right\vert
+1}\right) }\underset{\left\{ n\in \mathbb{A}_{0,2}:\text{ }M_{\left\vert
\alpha _{k}\right\vert }<n<M_{\left\vert \alpha _{k}\right\vert +1}\right\} }%
{\sum }\frac{1}{M_{\left\vert \alpha _{k}\right\vert +1}^{p}} \\
&\geq &\frac{cM_{\left\vert \alpha _{k}\right\vert }^{1-p}}{\Phi
^{1/2}\left( M_{\left\vert \alpha _{k}\right\vert +1}\right) }\rightarrow
\infty ,\text{ when \ \ }k\rightarrow \infty .
\end{eqnarray*}%
Theorem 2 is proved.


\begin{thebibliography}{99}
\bibitem{AVD} G. N. AGAEV, N. Ya. VILENKIN, G. M. DZHAFARLY and A. I.
RUBINSHTEIN, Multiplicative systems of functions and harmonic analysis on
zero-dimensional groups, Baku, Ehim, (1981) (in Russian).

\bibitem{blahota} I. BLAHOTA, Relation between Dirichlet kernels with
respect to Vilenkin-like systems, Acta Academiae Paedagogicae Agriensis,
XXII, (1994), 109-114.

\bibitem{BGG} I. BLAHOTA, G. GÁT and U. GOGINAVA, Maximal operators of
Fej\'er means of double Vilenkin-Fourier series, Colloq. Math. 107, no. 2,
(2007), 287-296.

\bibitem{BGG2} I. BLAHOTA, G. GÁT and U. GOGINAVA, Maximal operators of
Fej\'er means of Vilenkin-Fourier series. JIPAM. J. Inequal. Pure Appl.
Math. 7, (2006), 1-7.

\bibitem{Fu} N. J. FUJII, A maximal inequality for $H_{1}$ functions on the
generalized Walsh-Paley group, Proc. Amer. Math. Soc. 77, (1979), 111-116.

\bibitem{gat} G. GÁT, Ces\`{a}ro means of integrable functions with respect
to unbounded Vilenkin systems. J. Approx. Theory 124, no. 1, (2003), 25-43.

\bibitem{gat1} G. GÁT, Investigations of certain operators with respect to
the Vilenkin system, Acta Math. Hung., 61, (1993), 131-149.

\bibitem{GoAMH} U. GOGINAVA, Maximal operators of Fej\'er means of double
Walsh-Fourier series. Acta Math. Hungar. 115, no. 4, (2007), 333-340.

\bibitem{GoPubl} U. GOGINAVA, The maximal operator of the Fejér means of the
character system of the $p$-series field in the Kaczmarz rearrangement.
Publ. Math. Debrecen 71, no. 1-2,(2007), 43-55.

\bibitem{GNCz} U. GOGINAVA and K. NAGY , On the maximal operator of
Walsh-Kaczmarz-Fejér means, Czechoslovak Math. J. (to appear).

\bibitem{PS} J. PÁL and P. SIMON, On a generalization of the concept
ofderivate, Acta Math. Hung., 29 (1977), 155-164.

\bibitem{Sc} F. SCHIPP, Certain rearrangements of series in the Walsh
series, Mat. Zametki, 18, (1975), 193-201.

\bibitem{SiW} P. SIMON, F. WEISZ, Strong convergence theorem for
two-parameter Vilenkin-Fourier series. Acta Math. Hungar. 86, (2000), 17-38

\bibitem{Si2} P. SIMON, Investigations with respect to the Vilenkin system,
Annales Univ. Sci. Budapest E\"otv., Sect. Math., 28, (1985) 87-101.

\bibitem{Si3} P. SIMON, Strong convergence of certain means with respect to
the Walsh-Fourier series, Acta Math. Hung., 49 (1-2), (1987), 425-431.

\bibitem{Si4} P. SIMON, Strong Convergence Theorem for Vilenkin-Fourier
Series, Journal of Mathematical Analysis and Applications, 245, (2000),
52-68.

\bibitem{sm} B. SMITH, A strong convergence theorem for $H^{1}\left(
T\right) ,$ in Lecture Notes in Math., 995, Springer, Berlin, (1994),
169-173.

\bibitem{tep1} G.TEPHNADZE, Fej\'er means of Vilenkin-Fourier series, Stud.
sci. Math. Hung., 49 (1), (2012) 79-90.

\bibitem{tep2} G.TEPHNADZE, On the maximal operator of Vilenkin-Fej\'er
means, Turk. J. Math, 37, (2013), 308-318.

\bibitem{tep3} G.TEPHNADZE, On the maximal operators of Vilenkin-Fej\'er
means on Hardy spaces, Mathematical Inequalities \& Applications Vol. 16, N.
2, (2013), 301-312.

\bibitem{tep4} G.TEPHNADZE,\textit{\ }A note on the Fourier coefficients and
partial sums of Vilenkin-Fourier series, Acta Math. Acad. Paed. Ny\'\i reg.
(AMAPN), 28, (2012), 167-176.

\bibitem{tep5} G.TEPHNADZE, Strong convergence theorems of Walsh-Fej\'er
means, Acta Math. Hungar, (to appear).

\bibitem{tep6} G.TEPHNADZE, On the partial sums of Vilenkin-Fourier series,
Journal of Contemporary Mathematical Analysis, (to appear).

\bibitem{Vi} N. Ya. VILENKIN, On a class of complete orthonormal systems,
Izv. Akad. Nauk. U.S.S.R., Ser. Mat., 11 (1947), 363-400.

\bibitem{We1} F. WEISZ, Martingale Hardy spaces and their applications in
Fourier Analysis, Springer, Berlin-Heideiberg-New York, (1994).

\bibitem{We3} F. WEISZ, Hardy spaces and Ces\`{a}ro means of two-dimensional
Fourier series, Bolyai Soc. Math. Studies, (1996), 353-367.

\bibitem{We2} F. WEISZ, Ces\`{a}ro summability of one and two-dimensional
Fourier series, Anal. Math. 5 (1996), 353-367.
\end{thebibliography}
\end{document}